\documentclass[12pt]{article}
\usepackage[english]{babel}
\usepackage{amssymb}
\usepackage{amsmath}
\usepackage{amsmath,amsthm,amsfonts}
\usepackage{authblk}

\theoremstyle{plain}
\newtheorem{theorem}{Theorem}
\theoremstyle{plain}
\newtheorem{lemma}{Lemma}
\theoremstyle{plain}

\theoremstyle{plain}

\theoremstyle{plain}
\newtheorem{definition}{Definition}
\theoremstyle{plain}
\newtheorem{remark}{Remark}

\newtheorem{question}{Question}

\begin{document}

\title{A representing system generated by the Szeg\"{o} kernel for the Hardy space}

\author{Konstantin S. Speransky, Pavel A. Terekhin}

\affil{Saratov State University, Russia\\
konstantin.speransky@gmail.com, terekhinpa@mail.ru}

\date{}

\maketitle

\begin{abstract}
In this paper we give an explicit construction of a representing system generated by the Szeg\"{o} kernel for the Hardy space. Thus we answer an open question posed by Fricain, Khoi and Lef\`evre. We use frame theory to prove the main result.

\end{abstract}

\section{The main result}

The aim of this paper is to give an affirmative answer to the following question raised by Fricain, Khoi and Lef\`evre \cite{FKL}. Here and subsequently, $K_{\lambda_n}(\cdot)=K(\cdot,\lambda_n)$, $n=1,2,\dots$ is a sequence generated by the Szeg\"{o} kernel $K(z,\lambda)=(1-\overline{\lambda}z)^{-1}$.

\begin{question}\label{ST_q1}
Can we construct a sequence of points $\{\lambda_n\}_{n=1}^{\infty}$ in the open unit disk $\mathbb{D}$ so that $\{K_{\lambda_n}\}_{n=1}^{\infty}$ forms a representing system for $H^2(\mathbb{D})$?
\end{question}

It is well-known that the sequence $\{K_{\lambda_n}\}_{n=1}^{\infty}$ is not a basis for the Hardy space $H^2(\mathbb{D})$ for any set of points $\{\lambda_n\}_{n=1}^{\infty}$. Moreover, the normalized sequence $\{(1-|\lambda_n|^2)^{1/2}K_{\lambda_n}\}_{n=1}^{\infty}$ can not be a Duffin-Schaeffer frame for $H^2(\mathbb{D})$. Nevertheless, we will show that question \ref{ST_q1} has a positive answer.

Let $\{\lambda_n\}_{n=1}^{\infty}$ be a sequence of points on the unit disk $\mathbb{D}$.
We partition $\{\lambda_n\}_{n=1}^{\infty}$ into groups, so each group consists of $k^{th}$ roots of unity placed on a circle with a radius $r_k=1-\frac1k$
\begin{equation}\label{ST_eq2}
\lambda_n=\lambda_{k,j}=(1-\tfrac1k)e^{\frac{2\pi ij}{k}}, \qquad j=0,\dots,k-1, \qquad k=1,2,\dots.
\end{equation}

\begin{theorem}\label{ST_t1}
Let $\{\lambda_n\}_{n=1}^{\infty}$ be given by \eqref{ST_eq2}.
Then the sequence $\{K_{\lambda_n}\}_{n=1}^{\infty}$ is a representing system for the Hardy space $H^2(\mathbb{D})$.
\end{theorem}

Note that any representing system is obviously a complete sequence. By Szeg\"{o}'s zero-point  theorem, the completeness of the sequence $\{K_{\lambda_n}\}_{n=1}^{\infty}$ is equivalent to the Blaschke condition being false, i.e.
\begin{equation}\label{ST_eq3}
\sum_{n=1}^{\infty}(1-|\lambda_n|)=\infty.
\end{equation}

At the same time, by the recovery theorem of Totik \cite{Tot}, if condition \eqref{ST_eq3} holds then there exist polynomials $P_{n,k}$, where $k=1,\dots,n$ and $n=1,2,\dots$, such that for every $f\in H^2(\mathbb{D})$ we have
$$
f=\lim_{n\to\infty}\sum_{k=1}^nf(\lambda_k)P_{n,k}.
$$

Of course, this approximation does not provide the representation by series of the form
$$
f=\sum_{n=1}^{\infty}x_nK_{\lambda_n}.
$$

Speaking informally, we must take ``too many'' points $\{\lambda_n\}_{n=1}^{\infty}$ to get the representation above and \eqref{ST_eq2} is one such choice of points.

\section{Representing systems and frames}\label{ST_s3}

This section is devoted to the study of the relationship between representing systems and frames.

\begin{definition}\label{ST_d1}
A sequence $\{f_n\}_{n=1}^{\infty}\subset F$ is a representing system for a Banach space $F$ if for every $f\in F$ there are coefficients $x_n$, $n=1,2,\dots$ such that
$$
f=\sum_{n=1}^{\infty}x_nf_n.
$$
\end{definition}

The following notion of a frame for a Hilbert space was introduced by Duffin and Schaeffer \cite{DS}.

\begin{definition}\label{ST_d2}
A sequence  $\{f_n\}_{n=1}^{\infty}\subset H\setminus\{0\}$ is a Duffin-Schaeffer frame for a Hilbert space $H$ if there exist constants $0<A\le B<\infty$ such that for all $f\in H$
$$
A\|f\|_H^2\le\sum_{n=1}^{\infty}|\langle f,f_n\rangle|^2\le B\|f\|_H^2.
$$
\end{definition}

It is known that every Duffin-Schaeffer frame $\{f_n\}_{n=1}^{\infty}$ is a representing system for a Hilbert space $H$
(see \cite[Theorem 5.1.6]{Chr}).

Usually, when the notion of a Duffin-Schaeffer frame is generalized to the case of a Banach space $F$, the duality $\langle f,g_n\rangle$ is considered as the values of the functionals $g_n\in G:=F^*$ at $f\in F$. Using this approach, the notions of an atomic decomposition and a Banach frame were introduced by Gr\"{o}chenig \cite{Gro}.

For our purposes, it is more convenient to introduce the dual definitions by considering the Fourier coefficients $\langle f_n,g\rangle$ of a functional $g\in G$ with respect to a sequence $\{f_n\}_{n=1}^{\infty}$ in the original space $F$.

Let $X$ be a sequence Banach space with a natural basis $\{e_n\}_{n=1}^{\infty}$ ($e_n=\{\delta_{mn}\}_{m=1}^{\infty}$ where $\delta_{mn}$ is the Kronecher delta). Therefore, the dual space $X^*$ is isomorphic to some sequence Banach space $Y$.

As before, $F$ is a Banach space and $G:=F^*$ is its dual space.

\begin{definition}\label{ST_d4}
We say that a sequence $\{f_n\}_{n=1}^{\infty}\subset F\setminus\{0\}$ of elements of a Banach space $F$ is a frame for $F$ with respect to $X$ if there exist constants
$0<A\le B<\infty$ such that for all bounded linear functionals $g\in G$ the following inequalities are satisfied
\begin{equation}\label{ST_eq4}
A\|g\|_G\le\|\{\langle f_n,g \rangle\}_{n=1}^{\infty}\|_Y\le B\|g\|_G.
\end{equation}
\end{definition}

If we take $F=G=H$ to be a Hilbert space and $X=Y=\ell^2$, we get a Duffin-Schaeffer frame.

\begin{lemma}\label{ST_l2}
A sequence $\{f_n\}_{n=1}^{\infty}\subset F\setminus\{0\}$ is a frame for $F$ with respect to $X$ if and only if the following two assertions hold

(i) for all $x\in X$ the series $\sum_{n=1}^{\infty}x_nf_n$ converges in $F$,

(ii) for all $f\in F$ there is an $x\in X$ such that $f=\sum_{n=1}^{\infty}x_nf_n$.
\end{lemma}

In particular, any frame is a representing system. The converse is also true: any representing system $\{f_n\}_{n=1}^{\infty}\subset F\setminus\{0\}$ is a frame for $F$ with respect to its coefficients space $X(f_n)$ consisting of all sequences $\{x_n\}_{n=1}^{\infty}$ for which the series $\sum_{n=1}^{\infty}x_nf_n$ converges in $F$. The coefficients space $X(f_n)$ is equipped with the norm
$$
\|x\|_{X(f_n)}=\sup_{N=1,2,\dots}\biggl\|\sum_{n=1}^Nx_nf_n\biggr\|_F.
$$

In general, the same representing system can be a frame with respect to various sequence spaces $X$.

Definition \ref{ST_d4} was introduced in \cite{Ter1}. For more details about proofs in this section we refer the reader to \cite{Ter2} and \cite{Ter3}.

\section{Proof of the theorem}\label{ST_s4}

We have divided the proof into a sequence of lemmas.

\begin{lemma}\label{ST_l3}
For each $k\in\mathbb{N}$, let $\omega_{k}^j$ be a $k^{th}$ root of unity
$$
\omega_{k}^j=e^{\frac{2\pi ij}{k}}, \qquad j=0,\dots,k-1,
$$
and
$$
\|f\|_k:=\biggl(\frac1k\sum_{j=0}^{k-1}|f(\omega_k^j)|^2\biggr)^{1/2}.
$$

For each polynomial $P(z)=\sum_{j=0}^{k-1}c_jz^j$ of degree less than $k$ one has
$$
\|P\|_k=\|P\|_{H^2}.
$$
\end{lemma}

\begin{proof}
As the inverse discrete Fourier transform
$$
\check{c}_{l}=\frac{1}{\sqrt{k}}\sum_{j=0}^{k-1}c_j\omega_k^{jl}=\frac{P(\omega_k^l)}{\sqrt{k}}, \qquad l=0,\dots,k-1,
$$
is unitary, it follows that
$$
\|P\|_k=\biggl(\sum_{l=0}^{k-1}|\check{c}_l|^2\biggr)^{1/2}=\biggl(\sum_{j=0}^{k-1}|c_j|^2\biggr)^{1/2}=\|P\|_{H^2}.
$$
\end{proof}

\begin{lemma}\label{ST_l4}
Let
$$
\sigma_rf(z)=f(rz), \qquad 0<r<1.
$$
The following inequality holds
$$
\|\sigma_rf\|_k\le\frac{\|f\|_{H^2}}{(1-r^{2k})^{1/2}}.
$$
\end{lemma}

\begin{proof}
We can expand $f\in H^2$, $f=\sum_{n=0}^{\infty}c_nz^n$ into an orthogonal series
$$
f(z)=\sum_{l=0}^{\infty}z^{kl}P_l(z),
$$
where $P_l(z)=\sum_{j=0}^{k-1}c_{j+kl}z^j$ is a polynomial of degree less than $k$.
Then
\begin{align*}
\|\sigma_rf\|_k\le\sum_{l=0}^{\infty}r^{kl}\|\sigma_rP_l\|_k
=\sum_{l=0}^{\infty}r^{kl}\|\sigma_rP_l\|_{H^2}
\le\sum_{l=0}^{\infty}r^{kl}\|P_l\|_{H^2}\\
\le\biggl(\sum_{l=0}^{\infty}r^{2kl}\biggr)^{1/2}\biggl(\sum_{l=0}^{\infty}\|P_l\|_{H^2}^2\biggr)^{1/2}
=\frac{\|f\|_{H^2}}{(1-r^{2k})^{1/2}}.
\end{align*}
\end{proof}

\begin{lemma}\label{ST_l5}
The following inequalities hold true for all $f\in H^2$
\begin{equation}\label{ST_eq5}
\|f\|_{H^2}\le\sup_{k\in\mathbb{N}}\|\sigma_{1-1/k}f\|_k\le\frac{\|f\|_{H^2}}{(1-e^{-2})^{1/2}}.
\end{equation}
\end{lemma}

\begin{proof}
Since $(1-\frac1k)^k\uparrow e^{-1}$, we can easily obtain the upper estimate by using lemma \ref{ST_l4}
$$
\|\sigma_{1-1/k}f\|_k\le\frac{\|f\|_{H^2}}{(1-(1-\frac1k)^{2k})^{1/2}}\le\frac{\|f\|_{H^2}}{(1-e^{-2})^{1/2}}.
$$

To prove the lower estimate, we initially check it for an arbitrary polynomial $P(z)$, $\text{deg}\,P=N$. According to lemma \ref{ST_l3}, we have
$$
\sup_{k\in\mathbb{N}}\|\sigma_{1-1/k}P\|_k\ge\sup_{k>N}\|\sigma_{1-1/k}P\|_{H^2}=\|P\|_{H^2}.
$$

Now assume that $f\in H^2$ is an arbitrary function and select a polynomial $P$ such that $\|f-P\|_{H^2}<\varepsilon$.
Using the triangle inequality and the proof above, we can obtain
\begin{align*}
\sup_{k\in\mathbb{N}}\|\sigma_{1-1/k}f\|_k\ge\sup_{k\in\mathbb{N}}\|\sigma_{1-1/k}P\|_k-\sup_{k\in\mathbb{N}}\|\sigma_{1-1/k}(f-P)\|_k\\
\ge\|P\|_{H^2}-\frac{\varepsilon}{(1-e^{-2})^{1/2}}\ge\|f\|_{H^2}-\varepsilon-\frac{\varepsilon}{(1-e^{-2})^{1/2}}.
\end{align*}
The proof is complete as $\varepsilon$ tends to $0$.
\end{proof}

Now we have all the ingredients to prove theorem 1.

Let us denote by $\ell^2_k$ a $k$-dimensional Hilbert space equipped with the norm
$$
\|c\|_2:=\biggl(\sum_{j=0}^{k-1}|c_j|^2\biggr)^{1/2}.
$$

Throughout the proof, $X$ stands for a space with a mixed norm
$$
X=\ell^1(\ell^2_k)=\biggl(\bigoplus_{k=1}^{\infty}\ell_k^2\biggr)_{\ell^1}.
$$

Let
$$
I:=\{(k,j):j=0,1,\dots,k-1 \quad \text{and} \quad k=1,2,\dots\}
$$
be an index set. So $X=\ell^1(\ell^2_k)$ is the space of families $x=\{x_{k,j}\}_{(k,j)\in I}$ such that
$$
\|x\|_{1,2}:=\sum_{k=1}^{\infty}\biggl(\sum_{j=0}^{k-1}|x_{k,j}|^2\biggr)^{1/2}<\infty.
$$

Clearly, the dual space of $X$
$$
Y=\ell^{\infty}(\ell_k^2)=\biggl(\bigoplus_{k=1}^{\infty}\ell_k^2\biggr)_{\ell^{\infty}}
$$
is the space of families $y=\{y_{k,j}\}_{(k,j)\in I}$ satisfying
$$
\|y\|_{\infty,2}:=\sup_{k\in\mathbb{N}}\biggl(\sum_{j=0}^{k-1}|y_{k,j}|^2\biggr)^{1/2}<\infty
$$
with the standard duality
$$
\langle x,y\rangle=\sum_{k=1}^{\infty}\sum_{j=0}^{k-1}x_{k,j}y_{k,j}.
$$

Inequality \eqref{ST_eq5} of lemma \ref{ST_l5} implies that the normalized sequence
$$
\widehat{K}_{\lambda_n}=(1-|\lambda_n|^2)^{1/2}K_{\lambda_n}, \qquad n=1,2,\dots,
$$
fulfills the frame inequalities
$$
A\|g\|_{H^2}\le\|\{\langle\widehat{K}_{\lambda_n},g\rangle\}_{n=1}^{\infty}\|_{\infty,2}\le B\|g\|_{H^2}
$$
when $\{\lambda_n\}_{n=1}^{\infty}$ is defined by \eqref{ST_eq2}, because in this case we have
$$
(1-|\lambda_{k,j}|^2)^{1/2}=(1-(1-\tfrac1k)^2)^{1/2}\asymp\frac{1}{\sqrt{k}}, \qquad k=1,2,\dots,
$$
and by definition
\begin{align*}
\sup_{k\in\mathbb{N}}\|\sigma_{1-1/k}g\|_k
=\sup_{k\in\mathbb{N}}\biggl(\frac1k\sum_{j=0}^{k-1}|g(\lambda_{k,j})|^2\biggr)^{1/2}
\asymp \|\{\langle\widehat{K}_{\lambda_n},g\rangle\}_{n=1}^{\infty}\|_{\infty,2}.
\end{align*}

Applying lemma \ref{ST_l2} we conclude that for each function $f\in H^2(\mathbb{D})$ there exist coefficients $x_n=x_{k,j}$ such that
$$
\sum_{k=1}^{\infty}\biggl(\sum_{j=0}^{k-1}|x_{k,j}|^2\biggr)^{1/2}<\infty
$$
and the representation is valid
$$
f=\sum_{n=1}^{\infty}x_n\widehat{K}_{\lambda_n}=\sum_{n=1}^{\infty}x_n(1-|\lambda_n|^2)^{1/2}K_{\lambda_n}.
$$
This completes the proof.

\begin{remark}\rm
Theorem \ref{ST_t1} was announced without proof in our paper \cite{ST}.
\end{remark}

\end{document}